\documentclass[a4paper,11pt]{article}
\usepackage{amsthm,amsmath,amsfonts,amssymb,hyperref}
\usepackage{microtype}
\usepackage{graphicx,color}
\usepackage{calrsfs}
\usepackage{color}

\numberwithin{equation}{section}
\newtheorem{thm}{Theorem}[section]
\newtheorem{pro}{Proposition}[section]

\newtheorem{lma}{Lemma}[section]
\theoremstyle{definition}
\newtheorem{dfn}{Definition}[section]

\newtheorem{rem}{Remark}[section]
\newtheorem{exa}{Example}[section]
\renewcommand{\d}{\mathrm{d}}

\newcommand{\N}{\mathbb{N}}
\newcommand{\R}{\mathbb{R}}

\newcommand{\E}{\mathbf{E}}

\begin{document}
\title{On modeling weakly stationary processes}
\author{Pauliina Ilmonen \and Lauri Viitasaari}

\renewcommand{\thefootnote}{\fnsymbol{footnote}}

\author{Lauri Viitasaari\footnotemark[1] \, and \, Pauliina Ilmonen\footnotemark[2]}

\footnotetext[1]{Department of Information and Service Management, Aalto University School of Business, P.O. Box 21210, FIN-00076 Aalto, Finland, \texttt{lauri.viitasaari@iki.fi}}

\footnotetext[2]{Department of Mathematics and Systems Analysis, Aalto University School of Science, P.O. Box 11100, FIN-00076 Aalto, Finland, \texttt{pauliina.ilmonen@aalto.fi}}

\maketitle

\begin{abstract}
\noindent
In this article, we show that a general class of weakly stationary time series can be modeled applying Gaussian subordinated processes. We show that, for any given weakly stationary time series $(z_t)_{z\in\N}$ with given equal one-dimensional marginal distribution, one can always construct a function $f$ and a Gaussian process $(X_t)_{t\in\N}$ such that $\left(f(X_t)\right)_{t\in\N}$ has the same marginal distributions and, asymptotically, the same autocovariance function as $(z_t)_{t\in\N}$. Consequently, we obtain asymptotic distributions for the mean and autocovariance estimators by using the rich theory on limit theorems for Gaussian subordinated processes. This highlights the role of Gaussian subordinated processes in modeling general weakly stationary time series. We compare our approach to standard linear models, and show that our model is more flexible and requires weaker assumptions.
\end{abstract}

{\small
\medskip

\noindent
\textbf{AMS 2010 Mathematics Subject Classification:} 60G10, 62M10, 60F05

\medskip

\noindent
\textbf{Keywords:}
weak stationarity, autocovariance function, Gaussian subordinated processes, estimation, central limit theorem
}




\section{Introduction}
Time series models are of great significance in numerous areas of applications, e.g. finance, climatology and signal processing, to name just a few. Central limit theorems play an important role in statistical inference.  However, due to dependencies, it is challenging to obtain central limit theorems under general time series models. Moreover, from practical point of view, obtaining central limit theorem is not enough. It is also important to study how fast the convergence takes place, i.e. how far one is from the limiting distribution. 

A simple generalisation of the classical central limit theorem is the central limit theorem for $M$-dependent sequence of random variables. That is, the elements in the sequence are independent, if their indices are far away from each other. For general time series with arbitrary dependence structure, the problem becomes more subtle, and it might happen that the limiting distribution is not Gaussian and/or that one has to use different scaling than the classical $\sqrt{T},$ where $T$ is the sample size. Thus, a natural approach to the problem is to study limiting distributions of properly scaled averages of stationary processes with a given autocovariance structure. What happens on the limit is dictated by the dependence structure of the time series. If the dependence is weak enough, then central limit theorem is obtained. See a recent book \cite{rio} for a comprehensive introduction to the topic and \cite{her} for functional central limit theorem. Another option is to impose mixing conditions. Limit theorems for strong mixing processes are studied, e.g. in \cite{ded-rio,dou-mas-rio,mer-pel}. However, specific mixing conditions are often more than difficult to verify.

If we consider stationary time series models, two general classes, linear processes and Gaussian subordinated processes, are applied extensively in different fields. The class of univariate linear processes consists of stationary processes $(z_t)_{t\in\N}$ of the form
$$
z_t = \sum_{j=-\infty}^\infty \phi_j \xi_{t-j},
$$
where the coefficients $\phi_j$ satisfy some specific assumptions and $(\xi_j)_{j\in\N}$ is a sequence of independent and identically distributed random variables. For example, this class covers stationary ARMA-models with I.I.D. errors. For theory of such processes together with central limit theorems we refer to \cite{bro-dav} as well as to more recent articles \cite{ho-hsi,wu} studying limit theorems of linear processes. Finally, we mention \cite{che-ho}, where Berry-Esseen type bounds are derived for linear processes and \cite{jac-et-al2,jac-et-al}, where estimation of the mean and the autocovariances is studied in the case of long-memory and heavy-tailed linear processes.

The class of univariate Gaussian subordinated processes consists of stationary processes $(z_t)_{t\in\N}$ of the form $z_t = f(X_t)$, where $(X_t)_{t\in\N}$ is a $d$-variate stationary Gaussian process and $f$ is a given function. It is usually assumed that $f(X_0) \in L^2$. Central limit theorems for such time series date back to Breuer and Major \cite{breuer-major} and the topic has been studied extensively. Indeed,
for Gaussian subordinated processes central and non-central limit theorems have been studied at least in \cite{arcones-clt,avr-taq,bai-taq2,bai-taq,dob-maj,gir-sur}. Motivated by real-life applications, the non-central limit theorems have been studied mostly in the case of long-memory processes. In this case one has to use stronger normalisation, and the limiting distribution is Gaussian only if the so-called Hermite rank of the function $f$ is 1. More generally, in this case, the properly scaled average of $z_t$ converges towards a Hermite process of order $k$, where $k$ is the Hermite rank of $f$. These central and non-central limit theorems have been considered in statistical applications for long-memory processes at least in \cite{deh-taq} (empirical process and U-statistics), \cite{deh-roo-taq} (change point tests), \cite{levy-et-al} (estimation of scale and autocovariances in Gaussian setup), and \cite{gir-taq} (Whittle estimator).

In addition to the study of long-memory case and non-central limit theorems, the central limit theorems for Gaussian subordinated stationary processes have emerged again to the center of mathematical community's interest. The reason behind this is that it has been observed that Stein's method and Malliavin calculus suit together admirably well --- straightforwardly giving new tools to study central limit theorems for Gaussian subordinated processes. For recent developments on the topic, we refer to the articles \cite{nou-pec2,nou-pec1} and to the monograph \cite{n-p}. Also, a stronger version of the Breuer-Major theorem was proven in \cite{ivan-et-al}. It was proven that, in addition to the convergence in distribution, the convergence towards normal random variable holds even in stronger topologies, such as Kolmogorov or Wasserstein distance. Moreover, the authors also provided Berry-Esseen type bounds in these metrics. Finally, we mention \cite{bar-sur}, where the result was generalised to cover non-stationary Gaussian fields.

In this article, we consider univariate weakly stationary time series $(z_t)_{t\in\N}$. We study the asymptotic behavior of the traditional mean and autocovariance estimators under the assumption of  equal one-dimensional marginal distributions\footnote{By one-dimensional marginal distributions we refer to the distributions of $z_t$ for fixed time indices $t$} of $(z_t)_{t\in\N}$. Our main contribution is to show that for any weakly stationary time series $(z_t)_{t\in\N}$ with some given autocovariance structure and with some given equal one-dimensional marginal distributions, one can always construct a univariate Gaussian process $(X_t)_{t\in\N}$ and a function $f$ such that $(f(X_t))_{t\in\N}$ has, asymptotically, the same autocovariance structure and the same one-dimensional marginal distributions as $(z_t)_{t\in\N}$. Relying on that, we complement the above mentioned works on limit theorems in the case of Gaussian subordination. There exists a rich literature on the topic, and we propose to model time series directly with $(f(X_t))_{t\in\N}$. In comparison to the above mentioned literature, where the model is assumed to be $(f(X_t))_{t\in\N}$, we start with a given weakly stationary time series with equal one-dimensional marginals, and we construct a function $f$ and a Gaussian process $(X_t)_{t\in\N}$ such that $(f(X_t))_{t\in\N}$ is a suitable model for $(z_t)_{t\in\N}$. We obtain limiting normal distributions for the traditional mean and autocovariance estimators for any time series within our model that has absolutely summable autocovariance function. This corresponds to the case with short memory. In addition, we show that within our model, as desired, the function $f$ does have Hermite rank equal to 1. Indeed, Hermite rank equal to 1 ensures that even in the long-memory case, the limiting distribution is normal. We also show that if the one-dimensional marginal distribution is symmetric, then the corresponding Hermite ranks for variance and autocovariance estimators are (essentially) equal to 2. As such, our model is particularly suitable for modeling long memory, in which case the exact knowledge on the Hermite ranks is crucially important. We compare our approach and results to the existing literature including comparison to the theory of linear processes that are covered by our model. Note that, our model is not limited to, but covers e.g. stationary ARMA-models. We observe that the assumptions that are usually posed in the literature for obtaining limiting normal distribution, are clearly stronger than the assumptions we require. For example, in the short memory case our assumption of summable covariance function is rather intuitive, as well as easily verified, compared to, e.g. complicated assumptions on the coefficients $\phi_j$ of linear processes. These results highlight the applicability of Gaussian subordinated processes in modeling weakly stationary time series.

The rest of the article is organised as follows. In section \ref{sec:prel} we recall some basic definitions and preliminaries on Gaussian subordination. In Section \ref{sec:results} we introduce and discuss our model. Section \ref{sec:calibration} is devoted to the study of the standard mean, variance, and autocovariance estimators in the framework of our model. In Section \ref{sec:conc} we give some concluding remarks and compare our approach to the existing literature.

\section{Preliminaries}
\label{sec:prel}
In this section we review some basic definitions and fundamental results that are later applied in Section \ref{sec:results}. We start by recalling the definition of weak stationarity.
\begin{dfn}
Let $(z_t)_{t\in\N}$ be a stochastic process. Then $(z_t)_{t\in\N}$ is \emph{weakly stationary} if for all $t,s\in\N$,
\begin{enumerate}
\item $\E z_t = \mu < \infty$,
\item $\E z_t^2 = \sigma^2 < \infty$, \ and
\item $Cov(z_t,z_s) =r(t-s)$ for some function $r$.
\end{enumerate}
\end{dfn}
\begin{dfn}
Let $(z_t)_{t\in \N}$ be stationary with autocovariance function $r$. 
\begin{enumerate}
\item The process $z$ is called \emph{short-range dependent}, if 
$$
\sum_{j=1}^\infty |r(j)| < \infty.
$$
\item The process $z$ is called \emph{long-range dependent}, if, as $|j|\to \infty$, we have 
\begin{equation}
\label{eq:long-range-asymptotics}
r(j) \sim |j|^{2H-2}
\end{equation}
for some $H\in \left(\frac12,1\right)$. Here the notation $g(j)\sim f(j)$ means that $\lim_{j\to \infty} \frac{g(j)}{f(j)} = C$ for some constant $C$.
\end{enumerate}
\end{dfn}
\begin{rem}
To achieve greater generality, it is customary in the literature to add some slowly varying function $L(j)$ on the asymptotic behaviour \eqref{eq:long-range-asymptotics} of $r(j)$. For the sake of simplicity of the presentation, we have omitted this factor in our definition. However, it is straightforward to check that all our results remain valid in the general case as well.
\end{rem}

We now recall Hermite polynomials and the Hermite ranks of functions.

The Hermite polynomials $H_k$ are defined recursively as follows:
$$
H_0(x)=1, H_1(x) = x, \text{ \ and \ } H_{k+1}(x) = xH_k(x) - kH_{k-1}(x).
$$
The $k$th Hermite polynomial $H_k$ is clearly a polynomial of degree $k.$ Moreover,
it is well-known that Hermite polynomials form an orthogonal basis of the Hilbert space of functions $f$ satisfying
$$
\int_{-\infty}^{\infty} [f(x)]^2 e^{-\frac{x^2}{2}}\d x < \infty,
$$
or equivalently, $\E [f(X)]^2< \infty,$ where $X\sim N(0,1)$. Every $f$ that belongs to that Hilbert space has a Hermite decomposition
\begin{equation}
\label{eq:hermite_decomposition}
f(x) = \sum_{k=0}^\infty \alpha_k H_k(x),
\end{equation}
and for $X \sim N(0,1)$, $Y\sim N(0,1),$ we have that
\begin{equation}
\label{eq:f-cov}
\E [f(X)f(Y)] = \sum_{k=0}^\infty k! \alpha_k^2 \left[Cov(X,Y)\right]^k.
\end{equation}

\begin{dfn}[Hermite rank]
Let $(X_t)_{t\in\N},$ $X_t=\left(X_t^{(1)}, X_t^{(2)}, ..., X_t^{(d)}\right),$ be a $d$-dimensional stationary Gaussian process. Let $f: \mathbb{R}^d \rightarrow \mathbb{R},$  $f(X_t) \in L^2$. The function $f$ is said to have Hermite rank $q$ with respect to $X_t,$ if  $\E\left[(f(X_t)-\E f(X_t))p_m(X_t)\right]=0$ for all polynomials $p_m: \mathbb{R}^d \rightarrow \mathbb{R}$ that are of degree $m\leq q-1,$ and if there exists a polynomial $p_q$ of degree $q$ such that $\E\left[(f(X_t)-\E f(X_t))p_q(X_t)\right]\neq 0$.
\end{dfn}

\begin{rem}
Note that the Hermite rank of a function $f$ is the smallest number $q\geq 1$ such that $\alpha_q \neq 0$ in decomposition \eqref{eq:hermite_decomposition}.
\end{rem}
Processes of form $f(X_t)$ are called Gaussian subordinated processes, and there exists a rich theory on the statistical inference for subordinated processes. It turns out that the Hermite rank plays a crucial role. This fact is already visible in the following Breuer-Major theorem \cite{breuer-major}.
\begin{thm}{\cite[Theorem 1]{breuer-major}}
\label{thm:breuer-major}
Let $(X_t)_{t\in\N},$ $X_t=\left(X_t^{(1)}, X_t^{(2)}, ..., X_t^{(d)}\right),$ be a $d$-dimensional stationary Gaussian process. Assume that $f: \mathbb{R}^d \rightarrow \mathbb{R},$  $f(X_t) \in L^2,$ has a Hermite rank $q \geq 1$. Denote
$$
r_X^{k,i}(\tau) = \E[X^{(k)}_{\tau} X^{(i)}_0].
$$
If
$$
\sum_{\tau=0}^\infty |r_X^{k,i}(\tau)|^q < \infty,\quad \forall k,i=1,2,\ldots,d,
$$
then $\sigma^2 = Var \left[f(X_0)\right] + 2\sum_{t=1}^\infty Cov[f(X_0),f(X_t)]$ is well-defined and
$$
\frac{1}{\sqrt{T}}\sum_{t=1}^T \left[f(X_t)-\E f(X_t)\right] \xrightarrow{d} N(0,\sigma^2),
$$
as $T \rightarrow \infty.$
\end{thm}
A stronger version of Theorem \ref{thm:breuer-major} was proven in a recent article \cite{ivan-et-al}. It was shown that the convergence holds even in stronger topologies than the convergence in distribution, e.g. the convergence holds in Wasserstein distance and in Kolmogorov distance. Furthermore, applying Theorem 2.1 of \cite{ivan-et-al}, it is possible to study the rate of convergence. Obviously, one could apply these results in our setting as well, but for a general function $f,$ the bounds are rather complicated. For an interested reader, we refer to \cite{ivan-et-al}.
It is also known \cite{nou-nua-funcBM} that, under the additional assumption that $f(X_t) \in L^{2+\epsilon}$ for some $\epsilon>0$, a functional version of Theorem \ref{thm:breuer-major} holds, i.e.
$$
\frac{1}{\sqrt{T}}\sum_{t=1}^{\lfloor nT\rfloor} \left[f(X_t)-\E f(X_t)\right]
$$
converges weakly towards $\sigma$ times a Brownian motion in the Skorokhod space.

The following result provides a generalisation into the long memory case, where the summability condition does not hold. For details, we refer to \cite{bai-taqqu} and the references therein.
\begin{thm}
\label{thm:long-memory-limit}
Assume that $f: \mathbb{R} \rightarrow \mathbb{R},$  $f(X_t) \in L^2,$ has a Hermite rank $q \geq 1$, and let $X$ be a stationary Gaussian process such that, as $t\to \infty$,  
$$
r_X(t)^q \sim |t|^{2H-2}
$$
for some $H\in \left(\frac12,1\right)$. Then 
$$
\frac{1}{T^{H-1}}\sum_{t=1}^T \left[f(X_t)-\E f(X_t)\right] \xrightarrow{d} Z_q,
$$
as $T \rightarrow \infty,$ where $Z_q$ is the so-called Hermite random variable of order $q$ multiplied with a constant.
\end{thm}
\begin{rem}
\label{rem:normalisation}
The normalisation in Theorem \ref{thm:long-memory-limit} stems from the fact that 
$$
Var\left(\frac{1}{T}\sum_{k=1}^T\left[f(X_t)-\E f(X_t)\right]\right) \sim \frac{1}{T}\sum_{k=1}^T r_X(t)^q.
$$
\end{rem}
\begin{rem}
We stress that $Z_1$ is just a normal random variable, and consequently the only difference compared to Theorem \ref{thm:breuer-major} is the normalisation. However, in the corresponding functional version, the limiting Gaussian process is the fractional Brownian motion instead of the standard Brownian motion.
\end{rem}
\section{Modeling general weakly stationary time series}
\label{sec:results}
Let $(z_t)_{t\in\N}$ be a given  weakly stationary univariate time series with an expected value $\mu=\E[z_t]$ and a given autocovariance function $r_z(\tau)=\E[z_{\tau} z_0]-\mu^2$. Without loss of generality and in order to simplify the presentation, we assume that $\mu=0$ and $Var(z_t)=1$. Assume that the one-dimensional marginals of $(z_t)_{t\in\N}$ are all equal. By equal one-dimensional marginal distributions we mean that the distribution of $z_t$ is the same for all time indices $t$. The corresponding one-dimensional variable is denoted by $z,$ and its cumulative distribution function is denoted by $F_z.$  

We begin with the following result stating that Gaussian subordinated processes can have arbitrary one-dimensional marginals. The claim is based on inverse sampling, and is rather widely accepted folklore in the Gaussian subordination literature. However, since in many textbooks the claim is stated only in the case of continuous distributions $F_z$, for the sake of clarity we present the proof. We stress that the proof is standard, and we do not claim originality here.
\begin{pro}
\label{pro:marginals}
Let $(z_t)_{t\in\N}$ be an arbitrary process with equal square integrable one-dimensional marginals $F_z$.
Then there exists a function $f$ and a standardised Gaussian process $(X_t)_{t\in\N}$ such that  $f(X_t) \in L^2$ has the same one-dimensional marginal distributions as the process $(z_t)_{t\in\N}$. In particular, $f$ has a Hermite decomposition
\begin{equation*}
f(x) = \sum_{j=0}^\infty \alpha_j H_j(x).
\end{equation*}
\end{pro}
\begin{proof}
Let $\Phi$ denote the distribution function of the standard normal distribution. Then, from the standard inverse method,
it is clear that $X = \Phi^{-1}\left(F_z(z)\right)$ has standardized normal distribution. Denote
$$
F_z^{-1}(y) = \inf_x\{F_z(x) \geq y\}
$$
the quantile function of $F_z$. It is well-known that $F_z^{-1}$ is the left-inverse of $F_z$, i.e. 
$F^{-1}_z[F_z(z)] = z$ almost surely. Thus $F_z^{-1}\left(\Phi(X)\right)$ is distributed as $z$, and hence we may set $f(\cdot) = F_z^{-1}(\Phi(\cdot))$.
Furthermore, since $z \in L^2$, we also have $f(X) \in L^2$. From this it follows also that $f$ has the Hermite decomposition. This concludes the proof.
\end{proof}
\begin{rem}
We emphasize that we are only claiming that the one-dimensional distributions of $\left(F^{-1}\left((\Phi(X_t)\right)\right)_{t\in\N}$ are equal to the one-dimensional distributions of $(z_t)_{t\in\N}$. The multidimensional distributions are not necessarily the same.
\end{rem}
By Proposition \ref{pro:marginals}, for any stationary process $z$ (with equal one-dimensional marginals) one can always choose $f$ such that $f(X)$ has the correct one-dimensional distributions. As the analysis of weakly stationary processes boils down to the analysis of the covariance, one would like to construct a Gaussian process $X$ such that, for a given sequence of coefficients $\alpha_k$, the process
$$
Z_t = \sum_{k=1}^\infty \alpha_k H_k(X_t)
$$
have also the same covariance structure than $(z_t)_{t\in \N}$. As Gaussian processes can have arbitrary covariance structures, this question can be rephrased whether each covariance function $r_z$ have a representation 
\begin{equation}
\label{eq:cov-representation}
r_z(\tau) = \sum_{k=1}^\infty k!\alpha_k^2 r_X(\tau)^k,
\end{equation}
where $r_X(\tau)$ is arbitrary covariance function. Unfortunately, this is not the case as the following example shows.
\begin{exa}
Let $Z_t = \frac{1}{\sqrt{6}}H_3(X_t)$. In order for \eqref{eq:cov-representation} to hold for arbitrary covariance function would require that every positive semidefinite matrix $R_Z$ has a representation
$$
R_Z = R_X \circ R_X \circ R_X,
$$
where $R_X$ is positive semidefinite as well and $\circ$ denotes the Hadamard, i.e. element-wise, product of matrices. This clearly does not hold for general matrices, and it is straightforward to construct counterexamples. For instance,
$$
R_Z = \begin{pmatrix}
1 & b^3 & 0\\
b^3 & 1 & b^3\\
0 & b^3 & 1
\end{pmatrix}
$$
with $\frac{1}{4}<b<\left(\frac14\right)^{\frac13}$ is positive definite, and leads to 
$$
R_X = \begin{pmatrix}
1 & b & 0\\
b & 1 & b\\
0 & b & 1
\end{pmatrix}
$$
which is not positive definite. 
\end{exa}
This example reveals that given the marginal distribution $F_z$ and the covariance $r_z$ of $z$, it might be that $F^{-1} _z(\Phi(X_t))$ does not have the same covariance than $z$. On the other hand, in many applications one is only interested in modeling long scale behaviour such as long range dependence. Luckily it turns out that for this purpose, $F^{-1}_z(\Phi(X_t))$ provides a good model. 
\begin{pro}
\label{prop:asymptotic-cov}
Suppose that $(z_t)_{t\in \N}$ is a long-range dependent stationary process with equal one-dimensional marginals. Then there exists a Gaussian process and a function $f$ such that the process $Z_t = f(X_t)$ has same one-dimensional marginals and 
$$
r_Z(t) \sim r_z(t)
$$
as $t \to \infty$.
\end{pro}
\begin{proof}
Again, we set $f(\cdot) = F^{-1}_z(\Phi(\cdot))$. Then the marginals of $Z$ are given by $F_z$. Moreover, using \eqref{eq:f-cov} we obtain
$$
r_Z(\tau) = \sum_{k=q}^\infty k!\alpha_k^2 r_X^k(\tau),
$$
where $q$ is the Hermite rank of $F^{-1}_z(\Phi(x))$. 
Since $r_z(\tau) \sim |\tau|^{2H-2}$, it remains to take any stationary Gaussian process that satisfies $r_X(\tau) \sim |\tau|^{\frac{2H-2}{q}}$. Indeed, it is clear that 
$$
\frac{r_Z(\tau)}{r_X^q(\tau)} \to q!\alpha_q^2
$$
whenever $r_X^k(\tau)>0$ and converges to zero.
Such Gaussian process clearly exists. 
\end{proof}
\begin{rem}
We can easily extend the result beyond long memory case, provided that the decay of $r_z$ is of certain type. Indeed, the statement holds whenever we can find a stationary Gaussian process with covariance given asymptotically as $r_z(\tau)^{\frac{1}{q}}$. For example, this is possible if $r_z$ has exponential decay.
\end{rem}
\begin{rem}
It is well-known that given the asymptotics of the autocovariance $r_X(\tau)$, the term $|r_X(\tau)|^q$ determines the asymptotics of the autocovariance of $\left(f(X_t)\right))_{t\in\N}$ \cite[p. 223]{beran-et-al}. We stress that here we do the opposite; given the autocovariance $r_z(\tau)$ we construct $(X_t)_{t\in\N}$ such that $\left(f(X_t)\right))_{t\in\N}$ has the autocovariance function $r_z(\tau) \sim |r_X(\tau)|^q$.
\end{rem}

\section{On model calibration}
\label{sec:calibration}
In this section we suppose that the process $(z_t)_{t\in \N}$ is given by 
\begin{equation}
\label{eq:gaussian_subordination}
z_t = f(X_t).
\end{equation}
In particular, motivated by Proposition \ref{pro:marginals} and Proposition \ref{prop:asymptotic-cov}, we consider the case $f(x) = F^{-1}(\Phi(x))$.

We are interested in the mean and the autocovariance estimators given by
$$
m_z = \frac{1}{T}\sum_{t=1}^T z_t,
$$
and
$$
\hat{r}_{z}(\tau) = \frac{1}{T}\sum_{t=1}^{T-\tau} \left[z_t-m_z\right]\left[z_{t+\tau}- m_z\right].
$$
For simplicity, we divide by $T$ instead of $T-\tau$. Consequently, the estimators $\hat{r}_{z}(\tau)$ are only asymptotically consistent. On the other hand, in this case, the sample autocovariance function preserves the desired property of positive semidefinitiness. Obviously the asymptotic behaviour of $\hat{r}_{z}(\tau)$ is the same as for
$$
\tilde{r}_{z}(\tau) = \frac{1}{T-\tau}\sum_{t=1}^{T-\tau} \left[z_t-m_z\right]\left[z_{t+\tau}- m_z\right].
$$
Finally, for the case $\mu=0,$ a simpler version
$$
\overline{r}_{z}(\tau) = \frac{1}{T}\sum_{t=1}^{T-\tau} z_tz_{t+\tau}
$$
is often used. If one is only interested in consistency of the estimator, the use of $\overline{r}_{z}(\tau)$ is justified by the following simple lemma which states that asymptotically the difference between $\hat{r}_{z}(\tau)$ and $\overline{r}_{z}(\tau)$ is negligible.
\begin{lma}
\label{lma:estimator_connection}
Assume that $m_z=m_z(T) \rightarrow \mu$ in probability, as $T \rightarrow \infty.$ Then
$$
\hat{r}_{z}(\tau) = \overline{r}_{z}(\tau) - [m_z(T)]^2 + O_p(T^{-1}).
$$
\end{lma}

\begin{proof}
We have that
\begin{equation*}
\begin{split}
\hat{r}_{z}(\tau) &= \frac{1}{T}\sum_{t=1}^{T-\tau}z_tz_{t+\tau} - m_z\frac{1}{T}\sum_{t=1}^{T-\tau}z_t - m_z \frac{1}{T}\sum_{t=1}^{T-\tau}z_{t+\tau} + m_z^2\frac{T-\tau}{T}\\
&= \overline{r}_{z}(\tau) - m_z^2 + R_T,
\end{split}
\end{equation*}
where
$$
R_T = \frac{m_z}{T}\sum_{t=T-\tau+1}^T z_t + \frac{m_z}{T}\sum_{t=1}^{\tau}z_t - \frac{\tau}{T}m_z^2
$$
for $\tau\geq 1$ and $R_T = 0$ for $\tau=0$. Now, since $(z_t)_{t\in\N}$ has finite second moments, the both sums in $R_T$ are bounded in probability. Similarly, the last term is $O_p(T^{-1}),$ as $T \rightarrow \infty$.
\end{proof}
The problem with the usage of $\overline{r}_z(\tau)$ instead of $\hat{r}_z(\tau)$ is in the rate of convergence that can play a crucial role under long memory. In order to study the rate of convergence (and possible limiting distributions) for autocovariance estimators one needs to study the Hermite rank of $g(X_t,X_{t+\tau})=f(X_t)f(X_{t+\tau})$ which, in general, can be larger or smaller than the rank of $f$. This fact is illustrated with the following simple examples.
\begin{exa}
Let $f(x)=x$. Then $f$ has Hermite rank 1, while $[f(x)]^2 = x^2$ has Hermite rank 2.
\end{exa}
\begin{exa}
Let $f(x) = H_2(x)$. Then $f$ has Hermite rank 2 as well as $[f(x)]^2 = x^4 - 2x^2 + 1$.
\end{exa}
\begin{exa}
Let $f(x) = H_3(x) + H_2(x)$. Then $f$ has Hermite rank $2$, while
$[f(x)]^2= x^6 + 2x^5 - 5x^4 - 8x^3 + 7x^2 +6x + 1$ has Hermite rank 1.
\end{exa}
More generally, for an arbitrary pair $(q,p) \in \N^2$ it is straightforward to construct examples of $f$ where $f$ has rank $q$ and $f^2$ has rank $p$. In view of Remark \ref{rem:normalisation}, this means that the mean estimator $m_z$ is of order
$$
m_z = O_p\left(\sqrt{\frac{1}{T}\sum_{k=1}^T r_X(t)^q}\right)
$$
while the variance estimator is of order
$$
\overline{r}_z(0) = O_p\left(\sqrt{\frac{1}{T}\sum_{k=1}^T r_X(t)^p}\right).
$$
Thus the asymptotic properties of the estimators $\hat{r}_z(0)$ and $\overline{r}_z(0)$ can be very different. Similarly, one can construct examples of $f$ where 
the rank of $f^2$ is $q$ and the (two-dimensional) rank of $f(X_\tau)f(X_0)$, for fixed $\tau$, is $p$. Thus even the asymptotical properties and the rate of convergence for variance estimator $\overline{r}_z(0)$ and autocovariance $\overline{r}_z(\tau)$ can be different, and it is crucially important to have knowledge on the exact ranks of $f(X_0)$ and $f(X_\tau)f(X_0)$.
This is problematic, since in practice the function $f$ is usually not known. On the other hand, in our case we have $f(x)=F^{-1}(\Phi(x))$, where the quantile function $F^{-1}$ can be estimated from the observations. In this case it turns out that the Hermite rank is known as well.
\begin{pro}
\label{pro:rank1}
Let $F$ be an arbitrary distribution function with finite variance. Then
$$
f(\cdot) = F^{-1}\left(\Phi(\cdot)\right)
$$
has Hermite rank 1.
\end{pro}
\begin{proof}
In order to prove the claim we have to show that
$$
\E [f(X)X] \neq 0
$$
for $X \sim N(0,1)$. We have
\begin{equation*}
\begin{split}
&\int_{-\infty}^\infty F^{-1}\left(\Phi(x)\right)x e^{-\frac{x^2}{2}}\d x\\
&= \int_{-\infty}^0 F^{-1}\left(\Phi(x)\right)xe^{-\frac{x^2}{2}}\d x
+ \int_0^\infty F^{-1}\left(\Phi(x)\right)xe^{-\frac{x^2}{2}}\d x\\
&= -\int_0^\infty F^{-1}\left(\Phi(-x)\right)xe^{-\frac{x^2}{2}}\d x
+ \int_0^\infty F^{-1}\left(\Phi(x)\right)xe^{-\frac{x^2}{2}}\d x\\
&= \int_0^\infty \left[F^{-1}\left(\Phi(x)\right)-F^{-1}\left(\Phi(-x)\right)\right]xe^{-\frac{x^2}{2}}\d x.
\end{split}
\end{equation*}
Since $F$ is non-decreasing, also $F^{-1}$ is non-decreasing and hence
$$
F^{-1}\left(\Phi(x)\right)-F^{-1}\left(\Phi(-x)\right) \geq 0
$$
for all $x\geq 0$. Furthermore, the inequality is strict for large enough $x$, giving
$$
\E [f(X)X] = \int_0^\infty \left[F^{-1}\left(\Phi(x)\right)-F^{-1}\left(\Phi(-x)\right)\right]xe^{-\frac{x^2}{2}}\d x > 0.
$$
\end{proof}
\begin{rem}
Hermite rank 1 makes the mean and the autocovariance estimators stable, and one usually obtains Gaussian limits with suitable normalisations. For detailed discussion on the stability in the case of Hermite rank 1, we refer to \cite{bai-taqqu}. 
\end{rem}
\begin{rem}
We stress again that while $z = F^{-1}\left(\Phi(X)\right)$ has distribution $F$, in general it is not true that $F^{-1}\left(\Phi(X)\right)X$ is distributed as $zX$. For example, if $z=g(X)$ with suitable $g$ the distribution of $g(X)X$ is not the same as the distribution of $F_{g(X)}^{-1}\left(\Phi(X)\right)X$. A simple example of such case is $\chi^2(1)$ distribution, where $g(x)=x^2$ but $F_{X^2}^{-1}\left(\Phi(x)\right) \neq x^2$. Clearly, $g(X)$ has Hermite rank $2$ while $F_{X^2}^{-1}\left(\Phi(x)\right)$ has Hermite rank 1. This fact highlights our proposal to model $z$ with $z=F^{-1}\left(\Phi(x)\right)$ directly. It is also worth to note that if $g(x)$ is bijective, then the distributions of $F_{g(X)}^{-1}\left(\Phi(X)\right)X$ and $g(X)X$ are equal.
\end{rem}
Proposition \ref{pro:rank1} allows us to study asymptotic properties of the mean estimator $m_z$. Moreover, we get asymptotic properties also for variance and autocovariance estimators in the case of short memory. 
\begin{thm}
\label{thm:short-memory}
Let $(z_t)_{t\in\N}$ be given by
$$
z_t = F_z^{-1}\left(\Phi(X_t)\right)
$$
and $\E z_t^4 = c < \infty$. 
Assume further that
\begin{equation}
\label{eq:summable_cov}
\sum_{\tau=1}^\infty |r_z(\tau)|<\infty.
\end{equation}
Then
\begin{equation}
\label{eq:mean_convergence}
\sqrt{T}[m_z - \mu] \rightarrow N(0,\sigma^2)
\end{equation}
with $\sigma^2 = Var(z_0) + 2\sum_{\tau=1}^\infty r_z(\tau)$,
and for any $k\geq 0$
\begin{equation}
\label{eq:cov_convergence}
\sqrt{T}\left[\hat{r}_z(0) - r_z(0), \hat{r}_z(1)-r_z(1),\ldots,\hat{r}_z(k)-r_z(k)\right] \rightarrow N(0,\Sigma),
\end{equation}
where $\Sigma = (\Sigma_{ij}),\ i,j=0,1,\ldots,k$ is given by
$$
(\Sigma)_{ij} = Cov(z_0z_{i},z_0z_{j}) +2\sum_{\tau=1}^\infty Cov(z_\tau z_{i+\tau},z_0z_{j}).
$$
\end{thm}
\begin{proof}
The convergence \eqref{eq:mean_convergence} follows directly from Theorem \ref{thm:breuer-major} together with the fact that, by Proposition \ref{pro:rank1}, we have $r_z(t) \sim r_X(t)$. For the convergence \eqref{eq:cov_convergence}, first note that without loss of generality and for the sake of simplicity, we may and will assume that $\mu=0$ and use the estimators $\overline{r}_z(k)$ instead. Indeed, the general case then follows easily from \eqref{eq:mean_convergence}, Lemma \ref{lma:estimator_connection}, and the Slutsky's theorem. In order to prove \eqref{eq:cov_convergence} we have to show that, for any $n\geq 1$ and any $(\alpha_1,\ldots,\alpha_n)\in \R^n$, the linear combination
\begin{equation}
\label{eq:lin_combination}
\sqrt{T}\sum_{k=0}^n \alpha_k\left[\overline{r}_z(k)-r_z(k)\right],
\end{equation}
converges towards a Gaussian random variable. We define an $n+1$-dimensional stationary Gaussian process
$\overline{X}_t =(X_t,X_{t+1},\ldots,X_{t+n})$ and a function
$$
G(\overline{X}_t) = \sum_{k=0}^n \alpha_k \left[f(X_t)f(X_{t+k}) - r_z(k)\right],
$$
where $f(\cdot) = F_z^{-1}(\Phi(\cdot))$. With this notation we have
$$
\sqrt{T}\sum_{k=0}^n \alpha_k\left[\overline{r}_z(k)-r_z(k)\right] =
\frac{1}{\sqrt{T}}\sum_{t=1}^T G(\overline{X}_t) + R(T),
$$
where
$$
R(T) = \frac{1}{\sqrt{T}}\sum_{k=0}^n \alpha_k \sum_{t=T-k+1}^T \left[z_tz_{t+k} - r_z(k)\right].
$$
Since $\E z_t^4 = c < \infty$, it follows from Cauchy-Schwarz inequality that $F(\overline{X})\in L^2$. Thus assumption \eqref{eq:summable_cov} together with Theorem \ref{thm:breuer-major} implies that
$$
\frac{1}{\sqrt{T}}\sum_{t=1}^T G(\overline{X}_t) \rightarrow N(0,\sigma^2).
$$
For the term $R(T)$, we observe that the sum
$$
\sum_{k=0}^n \alpha_k \sum_{t=T-k+1}^T \left[z_tz_{t+k} - r_z(k)\right]
$$
is bounded in $L^2$, and hence $R(T) \rightarrow 0$ in probability. Thus the convergence of any linear combination of the form \eqref{eq:lin_combination} towards a normal random variable follows directly from Slutsky's theorem. Finally, the covariance matrix $\Sigma$ is derived by considering convergence of
$$
\sqrt{T}\left[\hat{r}_z(i)-r_z(i) + \hat{r}_{z}(j)-r_z(j)\right]
$$
together with Theorem \ref{thm:breuer-major} and by direct computations.
\end{proof}
In the presence of long memory, one needs to also compute the ranks of $[F_z^{-1}(\Phi(\cdot))]^2$ (for variance estimation) and $F^{-1}_z(\Phi(X_\tau))F^{-1}_z(\Phi(X_0))$ (for autocovariance estimation). Unfortunately, given a general $F_z$ these can be again arbitrary. It turns out however, that if the distribution $F_z$ is symmetric (around $0$), then we can always compute the corresponding ranks. 

Recall that a distribution $F$ is symmetric if $F(x) = 1-F(-x)$ for all $x\in \R$. This translates into 
$$
F^{-1}(y) = - F^{-1}(1-y), \quad y\in[0,1].
$$
In view of the symmetry of the normal distribution, this further implies
\begin{equation}
\label{eq:symmetry}
F^{-1}(\Phi(x)) = -F^{-1}(\Phi(-x)).
\end{equation}
\begin{pro}
Let $X\sim N(0,1)$ and let $F$ be an arbitrary symmetric distribution function with finite variance. Then;
\begin{enumerate}
\item For odd numbers $k\geq 1$ we have
$$
\E[F^{-1}(\Phi(X))X^k] >0.
$$
\item For even numbers $k\geq 0$ we have
$$
\E[F^{-1}(\Phi(X))X^k] =0
$$
\end{enumerate}
In particular,
$$
f(\cdot) = F^{-1}\left(\Phi(\cdot)\right)
$$
has Hermite rank 1 and a decomposition
\begin{equation}
F^{-1}\left(\Phi(X_t)\right) = \sum_{k\geq 1} \alpha_k H_k(X_t),
\end{equation}
where, for $j=0,1,2,\ldots$, we have $\alpha_{2j}=0$.
\end{pro}
\begin{proof}
Let $k$ be fixed. Computing as in the proof of Proposition \ref{pro:rank1}, we get
\begin{equation*}
\begin{split}
&\int_{-\infty}^\infty F^{-1}\left(\Phi(x)\right)x^k e^{-\frac{x^2}{2}}\d x\\
&= \int_{-\infty}^0 F^{-1}\left(\Phi(x)\right)x^ke^{-\frac{x^2}{2}}\d x
+ \int_0^\infty F^{-1}\left(\Phi(x)\right)x^ke^{-\frac{x^2}{2}}\d x\\
&= (-1)^k\int_0^\infty F^{-1}\left(\Phi(-x)\right)x^ke^{-\frac{x^2}{2}}\d x
+ \int_0^\infty F^{-1}\left(\Phi(x)\right)x^ke^{-\frac{x^2}{2}}\d x\\
&= \int_0^\infty \left[F^{-1}\left(\Phi(x)\right)+(-1)^k F^{-1}\left(\Phi(-x)\right)\right]x^ke^{-\frac{x^2}{2}}\d x.
\end{split}
\end{equation*}
As in the proof of Proposition \ref{pro:rank1}, this shows the claim for odd numbers $k$. Similarly, the claim for even $k$ follows from 
\eqref{eq:symmetry}.
\end{proof}
\begin{pro}
\label{pro:symmetric-rank2}
Let $F_z$ be symmetric and let $\tau \in \mathbb{Z}$ be fixed. Then the Hermite rank of $F_z^{-1}(\Phi(X_\tau))F_z^{-1}(\Phi(X_0))$ is at least $2$. Moreover, if $r(k)=r_X(k)$ is non-degenerate, i.e. for all $j\in \N$ we have $r(m)=r(j)$ for at most finitely many $m\in \N$, then the set
$$
S = \{\tau: F_z^{-1}(\Phi(X_\tau))F_z^{-1}(\Phi(X_0)) \textit{ has rank above two}\}
$$
is finite. In particular, if $r(k)\to 0$ as $k\to \infty$, then the set $S$ is finite. 
\end{pro}
\begin{proof}
From
$$
H_k(X)X = H_{k+1}(X) + kH_{k-1}(X)
$$
we obtain
$$
f(X_0)X_0 = \sum_{k\geq 1} \left[\alpha_k H_{k+1}(X_0) + \alpha_k k H_{k-1}(X_0)\right] = \alpha_1 + \sum_{k\geq 2} \left[\alpha_{k-1}+(k+1)\alpha_{k+1}\right]H_k(X_0).
$$
Here we have only even terms $H_{2k}$ while $f(X_t)$ consists of odd terms $H_{2k+1}$, 
giving
\begin{equation*}
\E\left[f(X_t)f(X_0)X_0\right] =0.
\end{equation*}
Thus $q>1$ meaning that the rank is at least two. Let us next prove that the set $S$ is finite. We first note that now
\begin{equation*}
\E\left[f(X_t)f(X_0)X_tX_0\right] = \alpha_1^2 + \sum_{k\geq 2}\left[\alpha_{k-1}+(k+1)\alpha_{k+1}\right]^2k!r^k.
\end{equation*}
Let $\tau_n \to \infty$ be an arbitrary sequence. Since all bounded sequences have a convergent subsequence, we may without loss of generality assume that $r(\tau_n) \to r\in [-1,1]$. Furthermore, without loss of generality we can assume $r(\tau_n) \neq r$. We now argue by contradiction and suppose that $S$ is not finite. Then, by passing to a subsequence if necessary, we can find a sequence $\tau_n \in S$ such that $\tau_n\to \infty$ and $r(\tau_n)\to r$, $r(\tau_n) \neq r$. Since $\tau_n \in S$, the Hermite rank of $f(X_{\tau_n})f(X_0)$ with $\E(X_{\tau_n}X_0) = r(\tau_n) \to r$ is $q>2$ for all $n$. This means that
$$
\E\left[f(X_{\tau_n})f(X_0)X_{\tau_n}X_0\right] = \E\left[f(X_{t_n})f(X_0)\right]r(\tau_n) = \sum_{k\geq 1} \alpha_k^2 k!r(\tau_n)^{k+1}.
$$
However, we can regard 
$$
g_1(r)=\alpha_1^2 + \sum_{k\geq 2}\left[\alpha_{k-1}+(k+1)\alpha_{k+1}\right]^2k!r^k
$$ 
and
$$
g_2(r)=\sum_{k\geq 1} \alpha_k^2 k!r^{k+1}
$$
as real-analytic functions. Consequently, since they coincides for all $r(\tau_n)$ converging to $r$, by the identity theorem we conclude that they are equal everywhere. In particular, this gives us
$$
g_1(0) = \alpha_1^2 = g_2(0) = 0
$$
which leads to a contradiction since, by Proposition \ref{pro:rank1}, we have $\alpha_1\neq 0$. This concludes the proof.
\end{proof}
\begin{rem}
Note that if for some $N$ we have $r(j)=0$ for all $j\geq N$, the statement is still valid while our assumption on the non-degeneracy of $r$ is violated. 
\end{rem}
Now it is straightforward to obtain the following result on the long memory case, analogous to Theorem \ref{thm:short-memory}.
\begin{thm}
Let $(z_t)_{t\in\N}$ be given by
$$
z_t = F_z^{-1}\left(\Phi(X_t)\right),
$$
where $F_z$ is symmetric and $\E z_t^4 = c < \infty$. 
Assume further that $z$ is long-range dependent for some $H\in\left(\frac12,\frac34\right)$. Then there exists a constant $\sigma^2>0$ and a positive semidefinite matrix $\Sigma$ such that 
\begin{equation}
\label{eq:mean_convergence-long}
T^{1-H}[m_z - \mu] \rightarrow N(0,\sigma^2)
\end{equation}
and, for any $k\geq 0$,
\begin{equation}
\label{eq:cov_convergence-long}
\sqrt{T}\left[\hat{r}_z(0) - r_z(0), \hat{r}_z(1)-r_z(1),\ldots,\hat{r}_z(k)-r_z(k)\right] \rightarrow N(0,\Sigma).
\end{equation}
\end{thm}
\begin{proof}
The convergence \eqref{eq:mean_convergence-long} follows from Theorem \ref{thm:long-memory-limit} and Proposition \ref{pro:rank1}, and the convergence \eqref{eq:cov_convergence-long} can be proved by following the proof of Theorem \ref{thm:short-memory} and exploiting the facts that, by Proposition \ref{pro:symmetric-rank2}, the rank is at least two, and that 
$$
\sum_{k=1}^\infty r_X(t)^2 < \infty
$$
for $H<\frac34$. The details are left to the reader.
\end{proof}
\begin{rem}
We remark that here we have used the convention that zero vector can be viewed as $N(0,\Sigma)$ distributed random variable with zero variance. This corresponds to the case when the ranks of $\hat{r}_z(j)-r_z(j)$ are above two for all $j\leq k$. Note also that, by Proposition \ref{pro:symmetric-rank2}, we always obtain a non-trivial limiting distribution by choosing $k$ large enough.
\end{rem}
\section{Discussion}
\label{sec:conc}
In this article, we argued why it is advantageous to model weakly stationary time series using Gaussian subordinated processes, especially in the case of long memory. Under our model, we are able to provide limit theorems for the standard mean and autocovariance estimators. Furthermore, even functional versions of the central limit theorems and Berry-Esseen type bounds in different metrics are available. In our modeling approach $(z_t)_{t\in\N} = \left(f(X_t)\right)_{t\in\N}$, the Hermite rank of the function $f$ is equal to 1. This is especially useful in the case of long memory processes as the limiting distribution is normal if and only if the Hermite rank of $f$ is equal to 1. For the variance and autocovariance estimators, we also proved that the corresponding Hermite ranks are (essentially) two provided that the distribution is symmetric. While in general one can always symmetrize the distribution, one might lose essential information on the transformation. This can be viewed as the price to pay in the trade where we gain more knowledge on the Hermite ranks, allowing to obtain precise asymptotic results for different estimators.

We end this paper by comparing our approach to the existing literature. Linear processes of the form
$$
z_t = \sum_{j=0}^\infty \phi_j \xi_{t-j},
$$
where $(\xi_t)_{t\in \mathbb{Z}}$ is an independent and identically distributed sequence, are widely applied models for stationary time series. To obtain central limit theorems for the mean and the autocovariance estimators, conditions on the coefficients $(\phi_j)_{j\in\mathbb{Z}}$ are required. A sufficient condition for obtaining central limit theorems is
\begin{equation}
\label{eq:coefficient_summability}
\sum_{j=0}^\infty |\phi_j| < \infty
\end{equation}
together with $\E \xi_t^4 < \infty$ (see Theorem 7.1.2. and Theorem 7.2.1. in \cite{bro-dav}). As the sequence $(\xi_t)_{t\in\mathbb{Z}}$ is independent and identically distributed, it follows that the one-dimensional marginals of the process are equal. Moreover, it is customary to pose assumptions for $(\phi_j)_{j\in\mathbb{Z}}$ giving exponential decay for the covariance. Consequently, such linear processes are covered by our modeling approach. Moreover, it is easy to see that $\E \xi_t^4 < \infty$ implies $\E z_t^4 < \infty$, and \eqref{eq:coefficient_summability} is strictly stronger than the assumption of absolutely summable autocovariance function. Thus our modeling approach is more flexible and requires weaker assumptions.

\bibliographystyle{plain}
\bibliography{bibli_v2}

\end{document}